\documentclass[12pt]{amsart}
\usepackage{latexsym,amssymb,amsmath,ulem,color}
\newcommand{\R}{{\mathbb R}}

\newcommand{\Z}{{\mathbb Z}}
\newcommand{\N}{{\mathbb N}}

\newcommand{\C}{{\mathbb C}}

\newcommand{\La}{{\Lambda}}

\newcommand{\LL}{{\mathcal {L}}}

\newtheorem{theorem}{Theorem}[section]
\newtheorem{lemma}[theorem]{Lemma}
\newtheorem{proposition}[theorem]{Proposition}
\newtheorem{conjecture}[theorem]{Conjecture}
\newtheorem{definition}[theorem]{Definition}
\begin{document}
\title{``Spectral implies Tiling" for Three Intervals Revisited}

\author{Debashish Bose}
\address{Debashish Bose: The Institute of Mathematical Sciences, India}
\email{dbosenow@gmail.com}

\author{Shobha Madan}
\address{Shobha Madan: Indian Institute of Technology Kanpur, India}
\email{madan@iitk.ac.in}


\begin{abstract}
\medskip

In \cite{BCKM} it was shown that ``Tiling implies Spectral'' holds
for a union of three intervals and the reverse implication was
studied under certain restrictive hypotheses on the associated
spectrum. In this paper, we reinvestigate the ``Spectral implies
Tiling'' part of Fuglede's conjecture for the three interval case.
We first show that the ``Spectral implies Tiling'' for two intervals
follows from the simple fact that two distinct circles have at most
two points of intersections. We then attempt this for the case of
three intervals and except for one situation are able to prove
``Spectral implies Tiling''. Finally, for the exceptional case, we
show a connection to a problem of generalized Vandermonde varieties.
\end{abstract}
\maketitle

\section{\bf{Introduction}}

We begin with the standard definitions and the statement of Fuglede's conjecture.

Let $\Omega$ and $T$ be  Lebesgue measurable subsets of $\R^d$ with finite positive measure. For $\lambda \in \R^d$, let $$e_{\lambda}(x):=|\Omega|^{-1/2}  e^{2 \pi i \lambda .x}{\chi}_{\Omega}(x),\,\,\, x\in \R^d .$$

\begin{definition}
$\Omega$ is said to be a {\bf $spectral$ $set$} if there exists a subset  $\Lambda \subset \R^d$ such that the set of exponential functions $E_{\Lambda}:=\{e_\lambda:\lambda \in \Lambda\}$ is an
orthonormal basis for the Hilbert space $L^2(\Omega)$. The set $\Lambda$ is said to be a {\bf $spectrum$} for $\Omega$ and the pair $(\Omega,\Lambda)$ is called a {\bf $spectral$ $pair$}.
\end{definition}

\begin{definition}
$T$ is said to be a {\bf $prototile$} if $T$ tiles $\R^d$ by translations; i.e., if there exists a subset  $\mathcal T \subset \R^d$ such that $\{T+t: t \in \mathcal T\}$ forms a partition a.e. of $\R^d$, where $T+t=\{x+t : x \in T\}$. The set $\mathcal T$ is said to be a {\bf $tiling$ $set$}
for $T$ and the pair $(T,\mathcal T)$ is called a $ tiling$ $pair$.
\end{definition}

The study of relationships between spectral and tiling properties of
sets began with the work of B. Fuglede \cite{Fug}, who proved the
following result:
\smallskip

\begin{theorem}(Fuglede~\cite{Fug})\label{fuglede} Let $\LL$ be a full rank lattice in $\R^d$ and let $\LL^*$ be the dual lattice. Then $(\Omega,\LL)$ is a tiling pair  if and only if $(\Omega, \LL^*)$ is a spectral pair.
\end{theorem}
\smallskip

In the same paper, Fuglede made the following conjecture, which is also known as the Spectral Set conjecture.
\smallskip

\begin{conjecture}(Fuglede's conjecture)
{\it A set $\Omega \subset \R^d$ is a spectral set if and only if $\Omega$ tiles $\R^d$ by translation.}
\end{conjecture}

This led to the study of spectral and tiling properties of sets. We refer the reader to \cite{BM} for  a survey and the present status of this problem.
\medskip

In one dimension, for the simplest case when $\Omega$ is a finite
union of intervals, the problem is open in both directions and only
the $2$-interval case has been completely resolved by Laba in
\cite{L1}, where she proved that the conjecture holds true.
\medskip

In \cite{BCKM} the case of three intervals was explored. It was
shown there that the ``Tiling implies Spectral'' part of Fuglede's
conjecture is true in this case,  and the reverse implication was
proved under some restrictive hypothesis on the associated spectrum.
\medskip

Recently in \cite{BM}, the authors have shown that any spectrum
associated with a spectral set which is a finite union of intervals
is periodic (see \cite{Kol} for a simplification of the proof). One
of the key ingredients in both proofs is an embedding of the
spectrum in a suitable vector space, equipped with an indefinite
conjugate linear form. In this note we develop these ideas to give
another proof of the ``Spectral implies Tiling'' part of Fuglede's
conjecture for two intervals and then we attempt this for the case
of three intervals. With the exception of one case, we are able to
conclude that the ``Spectral implies Tiling'' indeed holds. In the
last section, we show a connection of the exceptional case to a
question on the intersections of generalized Vandermonde varieties
restricted to the $3$-torus.

\section{Embedding $\Lambda$ in a vector space}

In this section we recall the embedding of the spectrum in a vector
space \cite{BM}.
\medskip

Consider the $2n$-dimensional vector space $\C^n\times\C^n$. We write its elements as $\underbar{v}=\left(v_1,v_2\right)$ with $v_1,v_2\in \C^n$. We define a conjugate linear form $\odot$ on $\C^n\times\C^n$ as follows: for $\underbar{v},\underbar{w}\in\C^n\times\C^n$, let
$$\underbar{v}\odot\underbar{w}:= \langle v_1,w_1\rangle -\langle v_2, w_2\rangle, $$ where $\langle\cdot,\cdot\rangle$ denotes the usual inner product on $\C^n$. Note that this conjugate linear form is degenerate, i.e., there exists $\underbar{v} \in \C^n\times\C^n$, $\underbar{v} \neq 0$ such that $\underbar{v}\odot\underbar{v}=0$. We call such a vector a {\it null-vector}. For example, every element of $\mathbb T^n \times \mathbb T^n$ is a null-vector.
\medskip

A subset $S\subseteq\C^n\times\C^n$ is called a set of {\it mutually null-vectors} if
 $\forall \,\,\underbar{v},\underbar{w}\in S$\textcolor{blue}{,} we have $\underbar{v}\odot\underbar{w}=0$.
\medskip

It is clear from the definition that elements of a set of mutually
null-vectors are themselves necessarily null-vectors. Any linear
subspace $V$ spanned by a set of mutually null vectors is itself a
set of mutually null-vectors and $dim(V)\leq n$.
\medskip

Now, suppose $\Omega=\cup_{j=1}^n \left[a_j, a_j+r_j\right)$ is a union of
 $n$ disjoint intervals with $a_1=0$ and $ |\Omega| = \sum_1^n r_j=1$. We
define a map $\varphi_{\Omega}$ from $\R$ to $\mathbb T^n \times
\mathbb T^n \subseteq \C^n \times \C^n$ by $$ x\rightarrow
\varphi_{\Omega}(x)=\left(\varphi_1(x); \varphi_2(x)\right),$$ where
$$ \varphi_1(x)=\left(e^{2\pi i (a_1+r_1) x}, e^{2\pi i (a_2+r_2)
x}, \dots, e^{2\pi i (a_n+r_n) x}\right)$$ $$ \varphi_2(x)=\left(1,
e^{2\pi i a_2 x}, \dots, e^{2\pi i a_n x}\right).$$

For a set $\La \subset \R$, the mutual orthogonality of the set of
exponentials $E_\La = \{e_\lambda: \lambda \in \La\}$ is equivalent
to saying that the set $\varphi_\Omega(\La) =
\{\varphi_\Omega(\lambda) ; \lambda \in \Lambda \}$ is a set of
mutually null vectors, and so  the vector space $V_\Omega(\La)$
spanned by $\varphi_\Omega(\La)$ has dimension at most $n$.
Therefore  if $(\Omega , \La)$ is a spectral pair, we can say that
$\La$ has a ``local finiteness property'', in the sense that there
exists a finite subset $\mathcal{B}=\left\{y_1,\dots,y_m\right\}
\subseteq\Lambda$, $m\leq n$ which determines  $\Lambda$  uniquely.
More precisely we have,

\begin{lemma}\label{local finite}
Let $(\Omega,\Lambda)$ be a spectral pair and let $\mathcal{B}\subseteq\Lambda$ be such that
$\varphi_{\Omega}(\mathcal B):=\left\{\varphi_{\Omega}(y):y\in \mathcal{B} \right\}$ forms a basis of $V_{\Omega}(\Lambda)$. Then $x\in\Lambda\,$ if and only if $\,\varphi_{\Omega}(x)\odot\varphi_{\Omega}(y)=0,\,\, \forall \, y\in\mathcal{B}$.
\end{lemma}

Next, we give a criterion for the periodicity of the spectrum.

\begin{lemma}\label{repeated}
Let $(\Omega,\Lambda)$ be a spectral pair. If $\exists\ \lambda_1, \lambda_2 \in \Lambda$ such that $\varphi_{\Omega}(\lambda_1)=\varphi_{\Omega}(\lambda_2)$, then $d=|\lambda_1-\lambda_2|\in\N$ and $\Lambda$ is $d$-periodic, i.e., $\Lambda= \{\lambda_1,\dots,\lambda_d\} +d \Z$.
\end{lemma}

\begin{proof}
Since $\varphi_{\Omega}(\lambda_1)=\varphi_{\Omega}(\lambda_2)$, we
have $\varphi_{\Omega}(d)=(1,\dots,1;1,\dots,1)$ and hence
$\varphi_{\Omega}(x+d)=\varphi_{\Omega}(x), \forall \, x\in\R$. Let
$\mathcal{B}\subseteq\Lambda$ be such that $\varphi_{\Omega}
(\mathcal B)$ is a basis of $V_{\Omega}(\Lambda)$. Then, whenever
$\lambda\in\Lambda$, we have
$\varphi_{\Omega}(\lambda+nd)\odot\varphi_{\Omega}(y)=\varphi_{\Omega}(\lambda)\odot\varphi_{\Omega}(y)=0,\forall
\,n \in \Z \ \mbox{and} \ \forall \, y\in\mathcal{B}$. Thus
$\lambda+d\Z\subseteq \Lambda$ and so $\Lambda$ is $d$-periodic. By
a simple application of Poisson summation formula we see that
$d\in\N$. But $\Lambda$ must have density $1$ (by Landau's density
theorem \cite{Landau}), so we conclude that $\Lambda=
\{\lambda_1,\dots,\lambda_d\} +d \Z$.
\end{proof}




\section {Spectral Implies Tiling for 2 intervals}\label{2int vs}

We will now use the ideas developed in the previous section to give a simple proof of the ``Spectral implies Tiling'' part of Fuglede's conjecture for a set which is a union of two intervals. See \cite{L1} for the original proof.
\medskip

Let $\Omega=[0,r] \cup [a,a+1-r]$, where $0<r<1$, $r<a$, and let $(\Omega,\Lambda)$ be a spectral pair. Without loss of generality, we may assume that $0 \in \Lambda$.
\medskip

Consider the map $\varphi_{\Omega}:\Lambda\rightarrow\C^2\times\C^2$
given by $$\varphi_{\Omega}(\lambda):=(e^{2 \pi i\lambda r}, e^{2
\pi i \lambda (a+1-r)}; 1, e^{2 \pi i\lambda a})$$ and let
$V_\Omega(\Lambda)$ be the subspace spanned by
$\varphi_{\Omega}(\Lambda)=\{\varphi_{\Omega}(\lambda):\lambda\in\Lambda\}$.
Then we have $dim (V_\Omega(\Lambda)) \leq 2$. We will now show that
in fact $dim V_\Omega(\Lambda) = 2$, unless $\Omega$ is degenerate,
i.e., $\Omega$ consists of a single interval of length $1$.
\medskip

First, observe that for $\lambda, \lambda' \in \R$, the vectors
$\varphi_{\Omega}(\lambda)$ and $\varphi_{\Omega}(\lambda')$ are
linearly dependent if and only if $\varphi_{\Omega}(\lambda) =
\varphi_{\Omega}(\lambda')$ (since the third coordinate in
$\varphi_{\Omega}(x)$ is $1 \,  \forall \,\, x \in \R$). Now if
$dim(V_\Omega(\Lambda))=1$, then $\forall \, \lambda \in \Lambda$,
$\varphi_{\Omega}(\lambda)=\varphi_{\Omega}(0)=(1,1; 1,1)$ and so
$\chi_\Omega(n\lambda) =0 \,\forall n \in \Z $ and thus Poisson
summation implies that $\lambda \in \Z$. Further, observe that
$\Lambda$ is actually a subgroup of $\Z$, therefore using Landau's
density criteria we get $\Lambda=\Z$. In particular, $1 \in
\Lambda$, and so $e^{2 \pi i r}=1$, which implies that $r=0$ or $1$,
and thus this is a degenerate case.
\medskip

Now let $\lambda_1=0, \lambda_2, \lambda_3$ be the first three
elements of $\Lambda \cap [0,\infty)$. We claim that
$\varphi_{\Omega}(\lambda_2) \neq \varphi_{_\Omega}(0)$. For if
$\varphi_{\Omega}(\lambda_2) = \varphi_{_\Omega}(0)$ then by Lemma
\ref{repeated}, $\Lambda$ is $\lambda_2-$periodic, and since by our
assumption $\lambda_2$ is the smallest positive element of
$\Lambda$, we get $\Lambda= \lambda_2 \Z $ and
$dim(V_{\Omega}(\Lambda))=1$, a contradiction. A similar argument
shows that $\varphi_{\Omega}(\lambda_2) \neq
\varphi_{_\Omega}(\lambda_3)$. Hence, we have two possible cases to
consider:

\begin{enumerate}
\item $\varphi_{\Omega}(0)=\varphi_{\Omega}(\lambda_3)$,
\item $\varphi_{\Omega}(0),\varphi_{\Omega}(\lambda_2),\varphi_{\Omega}(\lambda_3)$ are all distinct.
\end{enumerate}
\medskip

{\bf Case(1).} By Lemma \ref{repeated}, $\lambda_3=d \in \N $ and $\Lambda = d\Z \cup (\lambda_2 + d \Z)$. But $\Lambda$ must have density $1$, so $d=2$. Next, $\varphi_{\Omega}(2)=\varphi_{\Omega}(0)$ implies that $e^{2\pi i 2 a}= e^{2\pi i 2 r}=  e^{2\pi i 2(a+1-r)} =1 $, and so $a \in \Z/2$
and $r= 1/2$. That such an $\Omega$ tiles $\R$ is now easy to see.

\medskip

{\bf Case(2).} Suppose that $\varphi_{\Omega}(0)$, $\varphi_{\Omega}(\lambda_2)$, $\varphi_{\Omega}(\lambda_3)$ are all distinct. Then any two of these are linearly independent and form a basis of $V_\Omega(\Lambda)$.
\medskip

Let,
\begin{equation}
A=\left(\begin{array}{cccc}
1 & 1  & 1 & 1 \\
1 & e^{2 \pi i \lambda_2 a}  & e^{ 2 \pi i \lambda_2 r} & e^{2 \pi i \lambda_2(a+1-r)} \\
1 & e^{2 \pi i \lambda_3 a}  & e^{ 2 \pi i \lambda_3 r} & e^{2 \pi i \lambda_3(a+1-r)} \\ \end{array}\right)
\end{equation}
\medskip

Then we have $Rank(A)=2$, and in particular

\begin{equation}
\left|\begin{array}{ccc}
1 & 1  & 1  \\
1 & e^{2 \pi i \lambda_2 a}  & e^{ 2 \pi i \lambda_2 r} \\
1 & e^{2 \pi i \lambda_3 a}  & e^{ 2 \pi i \lambda_3 r} \\
\end{array}\right| = 0
\end{equation}
\medskip

Therefore,

\begin{equation}
\left|\begin{array}{cc}
e^{2 \pi i \lambda_2 a}-1 &  e^{2 \pi i \lambda_2 r}-1  \\
e^{2 \pi i \lambda_3 a}-1 &  e^{2 \pi i \lambda_3 r}-1  \\
\end{array}\right| = 0
\end{equation}
\medskip

So finally we get,

\begin{equation}\label{4}
(e^{2 \pi i \lambda_2 a}-1)( e^{2 \pi i \lambda_3 r}-1 )=(e^{2 \pi i \lambda_2 r}-1)(e^{2 \pi i \lambda_3 a}-1).
\end{equation}
\medskip

Put $e^{2 \pi i \lambda_2 a}-1 =\alpha$, and $e^{2 \pi i \lambda_2
r}-1 =\beta$ in equation (\ref{4}).
\medskip

If $\alpha=0$, then $e^{2 \pi i \lambda_2 r}=e^{2 \pi i \lambda_2(a+1-r)}=1$, and so $\varphi_{\Omega}(\lambda_2)=\varphi_{\Omega}(0)$ which is a contradiction. If $\beta=0$, we have  $e^{2 \pi i \lambda_2 a} = e^{2 \pi i \lambda_2(a+1-r)}$ and thus $\varphi_{\Omega}(\lambda_2)$ is of the form $(1, c; 1, c)$. Since the set $\{\varphi_{\Omega}(0),\varphi_{\Omega}(\lambda_2)\}$ generates $V_{\Omega}(\Lambda)$, all elements of $V_{\Omega}(\Lambda)$ are of this form. So that $\Lambda\subseteq \Z$ and is a subgroup of $\Z$.
Thus $\Lambda=\Z$, and $\Omega$ tiles $\R$ by $\Z$.
\medskip

So without loss of generality, let $\alpha, \beta \neq 0$, so that we have $$ \alpha( e^{2 \pi i \lambda_3 r}-1 )=\beta(e^{2 \pi i \lambda_3 a}-1).$$

Consider now two circles given by $C_1(t)=\alpha (e^{2 \pi i t}-1)$,
and $C_2(s)=\beta (e^{2 \pi i s}-1)$, $t,\, s \in [0,1]$. Both
circles pass through $0$. Further note that

\begin{eqnarray}\label{points}
C_1(\lambda_3r) = C_2(\lambda_3a),\\
C_1(\lambda_2r) = C_2(\lambda_2a).\label{points2}
\end{eqnarray}

We consider the various possibilities.
\medskip

First, if the two circles coincide, then they have the same radius and center i.e., $\alpha = \beta$ and so  $e^{2 \pi i \lambda_2 a}= e^{2 \pi i \lambda_2 r}$.
Thus $ \varphi_{\Omega}(\lambda_2)$ is of the form $(1, c; c, 1),$ and we conclude that $\Lambda=\Z$ as before.
\medskip

Next, we consider the case when the two circles $C_1(t),C_2(t)$ are
distinct. As mentioned above, both the circles pass through $0$ and
now there are two more points of intersection given by equations
(\ref{points}) and (\ref{points2}). But two distinct circles can
have at most two distinct points of intersection. If
$C_1(\lambda_2r) = C_2(\lambda_2a) = 0$ then
$\varphi_{\Omega}(0)=\varphi_{\Omega}(\lambda_2)$  and similarly if
$C_1(\lambda_3r) = C_2(\lambda_3a) = 0$, then
$\varphi_{\Omega}(0)=\varphi_{\Omega}(\lambda_3)$. By our assumption
these cases are not possible. Thus the only possibility is that
$C_1(\lambda_2 r) = C_2(\lambda_3 a) = C_1(\lambda_3r) =
C_2(\lambda_2a) = \alpha\beta$. Then $ e^{2 \pi i \lambda_2 a} =
e^{2 \pi i \lambda_3 a}$ and $e^{2 \pi i \lambda_2r} = e^{ 2 \pi i
\lambda_3r}$, i.e.,
$\varphi_{\Omega}(\lambda_2)=\varphi_{\Omega}(\lambda_3)$ which is
again not possible. This completes the proof.




\section {On 3 intervals}\label{3int vs}

In this section, we investigate the ``Spectral implies Tiling'' part of Fuglede's conjecture for three intervals, in the same spirit as in the previous section.
\medskip

Let $\Omega= [0,r]\cup[a,a+s]\cup[b,b+1-r-s]$ where $0<r,s,r+s<1$ and let $(\Omega,\Lambda)$ be a spectral pair. We will assume here that $0 \in \Lambda$.
Again we define the map $\varphi_{_\Omega}:\Lambda\rightarrow \C^3 \times\C^3$ by
$$\varphi_{_\Omega}(\lambda):=(e^{2 \pi i \lambda r}, e^{2 \pi i \lambda (a+s)}, e^{2 \pi i \lambda (b+1-r-s)}; 1, e^{2 \pi i \lambda a}, e^{2 \pi i \lambda b})$$
and let $V_\Omega(\Lambda)$ be the subspace spanned by
$\varphi_{_\Omega}(\Lambda)=\{\varphi_{_\Omega}(\lambda):\lambda\in\Lambda
\}$. We know that $dim(V_{_\Omega}(\Lambda))\leq 3$. As in the
$2$-interval case we will first show that if $\Omega$ is
non-degenerate, in the sense that the three intervals are disjoint
and have non zero length, then $dim(V_{_\Omega}(\Lambda))=3$.

\begin{proposition}
Let $\Omega$, as above, be a spectral set and let $\Lambda$ be a
associated spectrum. Then $dim(V_{_\Omega}(\Lambda))=3$.
\end{proposition}

\begin{proof}
If $dim(V_{_\Omega}(\Lambda))= 1$, we again get $\Lambda=\Z, \,\, r, s, 1-r-s \in \Z$, i.e., $\Omega$ consists of a single interval of length $1$ and this is a degenerate case.

\medskip
So let, if possible, $dim(V_{_\Omega}(\Lambda))=2$. Suppose $0 <
\lambda_2 < \lambda_3$ are the first three elements of $\Lambda \cap
[0,\infty)$. In the two cases $\varphi_{_\Omega}(0) =
\varphi_{_\Omega}(\lambda_2)$ or $\varphi_{_\Omega}(\lambda_2) =
\varphi_{_\Omega}(\lambda_3)$ we conclude that
$dim(V_{_\Omega}(\Lambda))=1$, and if
$\varphi_{_\Omega}(0)=\varphi_{_\Omega}(\lambda_3)$ we see easily
that $\Lambda=2\Z \cup (2\Z+\alpha)$ and $r,s\in \Z/2$, i.e.,
$\Omega$ is a union of $2$ intervals of length $1/2$ or is a single
interval of length $1$, and this case too is degenerate.

\medskip
Finally, suppose that
$\varphi_{_\Omega}(0),\varphi_{_\Omega}(\lambda_2),\varphi_{_\Omega} (\lambda_3)$ are all distinct. Define
\begin{equation}
A=\left(\begin{array}{cccccc}
1 & 1  & 1 & 1 & 1 & 1 \\
e^{2 \pi i \lambda_2 r} & e^{2 \pi i \lambda_2 (a+s)} & e^{2 \pi i
\lambda_2 (b+1-r-s)} & 1 & e^{2 \pi i \lambda_2 a} & e^{2 \pi i
\lambda_2 b}\\
e^{2 \pi i \lambda_3 r} & e^{2 \pi i \lambda_3 (a+s)} & e^{2 \pi i
 \lambda_3 (b+1-r-s)} & 1 & e^{2 \pi i \lambda_3 a} & e^{2 \pi i
\lambda_3 b}\\
\end{array}\right)
\end{equation}

Then, by our assumption, $Rank(A)=2$, and the rows of $A$ are
distinct. Since $\varphi_{_\Omega}(0)\neq\varphi_{_\Omega}
(\lambda_2)$ and
$\varphi_{_\Omega}(0)\odot\varphi_{_\Omega}(\lambda_2)=0$, at least
two entries in the second row of $A$ are different from $1$, i.e.,
$\exists \,\, i_1, i_2$ such that $A(2,i_1),A(2,i_2)\neq 1$.
Consider the $3 \times 3$ matrix constructed out of the $1$st,
$i_1$th and $i_2$th column of A, since $Rank(A)=2$ it is singular.
Hence we have,
\begin{equation}
\left|\begin{array}{ccc}
1 & 1  & 1  \\
1 & A(2,i_1) & A(2,i_2) \\
1 & A(3,i_1) & A(3,i_2) \\
\end{array}\right| = 0
\end{equation}
\medskip

Using the fact that $Rank(A)=2$, and $A(2,i_1),A(2,i_2)\neq 1$ we
argue as in the two interval case to conclude that the circles
$C_1(t)= (A(2,i_1)-1) (e^{2\pi i t} -1)$ and $C_2(t)= (A(2,i_2)-1)
(e^{2\pi i t} -1)$ coincide. Therefore, $A(2,i_2)= A(2,i_1)=
\alpha$, say. By choosing other columns of $A$, we see that the
coordinates of $\varphi_{\Omega}(\lambda_2)$ are either $1$ or
$\alpha$. But since
$\varphi_{_\Omega}(0)\odot\varphi_{_\Omega}(\lambda_2)=0$,
$\varphi_{_\Omega}(\lambda_2)$ is either of the form  $(1,1,\alpha;
1,1,\alpha)$ or $(1,\alpha,\alpha; 1,\alpha,\alpha)$ (up to suitable
permutations). Now
$\{\varphi_{\Omega}(0),\varphi_{\Omega}(\lambda_2) \} $ forms a
basis of $V_{\Omega}(\Lambda)$, so as before, we see that
$\Lambda=\Z$. But then one of the intervals has length $0$ or $1$,
and this is a degenerate case.
\end{proof}

So $Rank(A)=3$. Let $\lambda_1=0,\lambda_2,\lambda_3,\lambda_4$ be
the first 4 elements of $\Lambda \cap [0,\infty)$. Each of the cases
$\varphi_{\Omega}(\lambda_i)=\varphi_\Omega(\lambda_{i+1})$ or
$\varphi_{\Omega}(\lambda_i)=\varphi_\Omega(\lambda_{i+2})$ will
imply $\Omega$ is degenerate, i.e., one of the intervals has length
$0$. Now if $\varphi_{\Omega}(0)=\varphi_{\Omega}(\lambda_4)$, then
$\lambda_4=d$ and the spectrum is $d$-periodic, and by a density
argument we conclude $d=3$. It follows then, that this is the case
of three equal intervals i.e., $r=s=1/3$ and spectral implies tiling
follows by the result of \cite{Newman} (see \cite{BCKM} for a
proof). So, now it remains to consider the case
that
$\varphi_{\Omega}(0),\varphi_{\Omega}(\lambda_2),\varphi_{\Omega}(\lambda_3),\varphi_{\Omega}(\lambda_4)$
are all distinct and $Rank(A)=3$.

\medskip
To proceed further, we will use the result that the spectrum is
periodic \cite{BM}. Let $d$ be the smallest positive integer such
that $d\Z \subseteq \Lambda$. Let $V_{\Omega}(d\Z)$ denote the
linear space spanned by the image of the arithmetic projection $d\Z$
under the map $\phi_\Omega$. Now if $dim(V_{_\Omega}(d\Z))=3$ or
$2$, then by the results of \cite{BCKM}, Section 5, we get Spectral
implies Tiling.

\medskip
So without loss of generality, we may assume that
$dim(V_{\Omega}(d\Z))=1$, and that $\Lambda$ is $d$-periodic. There
are now two possible cases to consider:

\begin{enumerate}
    \item $dim(V_{\Omega}(\Lambda \setminus d\Z))= 2$
    \item $dim(V_{\Omega}(\Lambda \setminus d\Z)) = 3$.
\end{enumerate}

In the first case we are able to show that $d=3$, thus $\Omega$ is a
union of three equal interval and hence Spectral implies Tiling  as
before. It is the second case that remains inconclusive.
\medskip


{\bf Case(1).} We show that in this case  $d = 3$. Suppose not, and $d > 3$. Let $\Lambda \cap (0,d)=\{\lambda_2,\lambda_3,\dots,\lambda_d\}$. Since $d$ is the minimal period, $\varphi_{\Omega}(\lambda_2),\varphi_{\Omega}(\lambda_3),\varphi_{\Omega}(\lambda_4)$ are all distinct and since $dim(V_{\Omega}(\Lambda \setminus d\Z))=2$, $\{\varphi_{\Omega}(\lambda_2),\varphi_{\Omega}(\lambda_3), \varphi_{\Omega}(\lambda_4)\}$ is a linearly dependent set. Let
\begin{eqnarray*}
\varphi_{\Omega}(\lambda_2) & = &(\xi_1,\xi_2,\xi_3; 1,\xi_5,\xi_6)\\
\varphi_{\Omega}(\lambda_3) & = &(\rho_1,\rho_2,\rho_3; 1,\rho_5,\rho_6)\\
\varphi_{\Omega}(\lambda_4) & = &(\eta_1,\eta_2,\eta_3;
1,\eta_5,\eta_6)
\end{eqnarray*}

Now there exists $i,j$ such that $\xi_i\neq\rho_i$ and
$\xi_j\neq\rho_j$. By our assumption
\begin{equation}\left|\begin{array}{ccc}
1 & \xi_i  & \xi_j  \\
1 & \rho_i & \rho_j \\
1 & \eta_i & \eta_j \\
\end{array}\right| = 0
\end{equation}
and so,

\begin{equation}
\left|\begin{array}{ccc}
1 & \xi_i  & \xi_j  \\
0 & \rho_i-\xi_i & \rho_j-\xi_j \\
0 & \eta_i-\xi_i & \eta_j-\xi_j \\
\end{array}\right| = 0
\end{equation}
Thus we obtain,

\begin{equation}
(\rho_i-\xi_i)(\eta_j-\xi_j)=(\rho_j-\xi_j)(\eta_i-\xi_i)
\end{equation}
which we rewrite as
\begin{equation}
(\rho_i \bar{\xi_i}-1)(\eta_j
\bar{\xi_j}-1)=(\rho_j\bar{\xi_j}-1)(\eta_i\bar{\xi_i}-1)
\end{equation}
\medskip

Since $\xi_i\neq\rho_i$ and $\xi_j\neq\rho_j$ and
$dim(V_{\Omega}(\Lambda \setminus d\Z))=2$ we can exclude the two
possibilities that $\{\eta_i=\rho_i,\eta_j=\rho_j\}$ or that
$\{\eta_i=\xi_i,\eta_j=\xi_j\}$. Then, by the same argument with two
circles as at the end of section 3, we see that $\rho_i
\overline{\xi_i} =\rho_j \overline{\xi_j}= \alpha$. In particular
this would hold for any other index $j'$ such that $\xi_{j'} \neq
\rho_{j'}$. This implies that
$\varphi_{\Omega}(\lambda_3-\lambda_2)=(1,1,\alpha;1,1,\alpha)$ or
$(1,\alpha,\alpha;1,\alpha,\alpha)$ (or some suitable permutation).
But since $\varphi_{_\Omega}(\lambda):=(e^{2 \pi i \lambda r}, e^{2
\pi i \lambda (a+s)}, e^{2 \pi i \lambda (b+1-r-s)}; 1, e^{2 \pi i
\lambda a}, e^{2 \pi i \lambda b})$, we see that
$\lambda_3-\lambda_2 \in \Z$. We write $\lambda_3-\lambda_2 =k$, and
show that $k\Z \subset \Lambda$.
\medskip

Consider the first situation, namely $\varphi_\Omega(k) =
(1,1,\alpha;1,1,\alpha)$. Now
$\varphi_{\Omega}(0),\varphi_{\Omega}(\lambda_2),\varphi_{\Omega}
(\lambda_3)$ is a basis of $V_{\Omega}(\Lambda)$, and we are in the
case where $\varphi_{\Omega} (\lambda_3) = (\xi_1,\xi_2,\alpha\xi_3;
1,\xi_5,\alpha\xi_6)$. But both $\varphi_{\Omega} (\lambda_2)$ and
$\varphi_{\Omega} (\lambda_3)$ are null vectors, so $\xi_3 = \xi_6$.
Then, $\varphi_{\Omega}(kn)\odot
\varphi_{\Omega}(\lambda_i)=0,\,\,i=1,2,3$ for every $n \in \Z$.
Thus by Lemma \ref{local finite}, $k\Z \subseteq \Lambda$, so that
$\Lambda$ and $k < d$. But $d$ is the smallest positive integer with
this property, which is a contradiction. A similar argument works
for all other cases. Thus $d=3$, and $\Omega$ is a union of $3$
equal intervals, and Spectral implies Tiling follows.

\section{Generalized Vandermonde Matrix}

It remains now to consider the case when $dim(V_\Omega(d \Z))=1$ and
$dim(V_\Omega(\Lambda \setminus d\Z))=3$ where $d$ is the smallest
integer such that $d \Z$ is in the spectrum $\Lambda$. Note that
$dim(V_\Omega(d \Z))=1$ implies that $\Omega$ can be written as
$$\Omega=[0,k_1/d] \cup [l_2/d,(l_2+k_2)/d] \cup [l_3/d,(l_3+k_3)/d],$$
where $l_i,k_i \in \N$ and $k_1+k_2+k_3=d$.
\medskip

If $d=3$ then it is the case of three equal intervals in which case
we know that Fuglede's conjecture holds. By known results it is
possible to rule out the cases $d=4$ and 5 as well. Hence the
problem will be resolved if we can show that $d<6$. In any case
finding a bound on $d$ is desirable.
\medskip

Now if $d>3$, let $0=\lambda_1,\lambda_2,\lambda_3 < d $  be three
elements of $\Lambda$ such that
$\{\varphi_{\Omega}(0),\varphi_{\Omega}(\lambda_2),
\varphi_{\Omega}(\lambda_3)\}$ forms a basis of
$V_{\Omega}(\Lambda)$. By our assumption there exists $\lambda_4<d$
in $\Lambda$ such that $\varphi_{\Omega}(\lambda_2),
\varphi_{\Omega}(\lambda_3),\varphi_{\Omega}(\lambda_4)$ are
linearly independent. We construct the matrix
\begin{equation}
A=\left(\begin{array}{cccccc}
1 & 1  & 1 & 1 & 1 & 1 \\
e^{2 \pi i \frac{\lambda_2 k_1}{d}} & e^{2 \pi i \frac{\lambda_2
(l_2+k_2)}{d}} & e^{2 \pi i \frac{\lambda_2 (l_3+k_3)}{d}} & 1 &
e^{2 \pi i \frac{\lambda_2 l_2}{d}} & e^{2 \pi i
\frac{\lambda_2 l_3}{d}}\\
e^{2 \pi i \frac{\lambda_3 k_1}{d}} & e^{2 \pi i \frac{\lambda_3
(l_2+k_2)}{d}} & e^{2 \pi i \frac{\lambda_3 (l_3+k_3)}{d}} & 1 &
e^{2 \pi i \frac{\lambda_3 l_2}{d}} & e^{2 \pi i
\frac{\lambda_3 l_3}{d}}\\
e^{2 \pi i \frac{\lambda_4 k_1}{d}} & e^{2 \pi i \frac{\lambda_4
(l_2+k_2)}{d}} & e^{2 \pi i \frac{\lambda_4 (l_3+k_3)}{d}} & 1 &
e^{2 \pi i \frac{\lambda_4 l_2}{d}} & e^{2 \pi i
\frac{\lambda_4 l_3}{d}}\\
\end{array}\right)
\end{equation}

The rank of this matrix is $3$ i.e., the rows are linearly
dependent, hence each of its $4 \times 4$ minors are zero.
\medskip

Observe that the $4\times 4$ minors are of the form
\begin{equation}
\left|\begin{array}{cccc}
1 & 1 & 1 & 1 \\
X_1^{i} & X_1^{j} & X_1^{k} & X_1^{l}\\
X_2^{i} & X_2^{j} & X_2^{k} & X_2^{l}\\
X_3^{i} & X_3^{j} & X_3^{k} & X_3^{l}\\
\end{array}\right|
\end{equation}

Thus we get after reductions, equations of the form

\begin{equation}\label{15}
\left|\begin{array}{cccc}
1 & 1 & 1 & 1 \\
1 & X_1^{j} & X_1^{k} & X_1^{l}\\
1 & X_2^{j} & X_2^{k} & X_2^{l}\\
1 & X_3^{j} & X_3^{k} & X_3^{l}\\
\end{array}\right|=0
\end{equation}

These are the determinants of generalized Vandermonde matrices in
the variables $(X_1,X_2,X_3)$ and exponents $(j,k,l)$. We write
(\ref{15}) as $$R_{(j,k,l)}(X_1,X_2,X_3)=0.$$

We are interested in the common zero solution set  of these
Vandermonde varieties intersected with the set $ 1 \times \mathbb
T^3$.
\medskip

In particular, let us consider the generalized Vandermonde matrix
which we get by taking those minors where the first three columns
correspond to the left end-points of the set $\Omega$ and the $4$th
column is one of the right end point i.e., a minor obtained by
choosing the $4$th, $5$th, and $6$th columns of the matrix $A$ and
one of the first three columns. Thus we consider $R_{(i_5,i_6,i_1)}(X_1,X_2,X_3)$,
$R_{(i_5,i_6,i_2)}(X_1,X_2,X_3),$ and $R_{(i_5,i_6,i_3)}(X_1,X_2,X_3).$
\medskip

In \cite{DZ} (Theorem 3.1) it is proved that the polynomials
$$ T_{(j,k,l)}(X_1,X_2,X_3) =\frac{R_{(j,k,l)}(X_1,X_2,X_3)}{V(X_1^g,X_2^g,X_3^g)},$$
 are either irreducible or constant. Here $g = gcd(j,k,l)$ and $V$ denotes the standard Vandermonde determinant, thus
$V(X_1^g,X_2^g,X_3^g)= R_{(1,2,3)}(X_1^g,X_2^g,X_3^g)$.
\medskip

Consider next, the Schur Polynomials given by
$$S_{(j,k.l)}(X_1,X_2,X_3)=\frac{R_{(j,k,l)}(X_1,X_2,X_3)}{V(X_1,X_2,X_3)},$$

Let $g_1=gcd(i_5,i_6,i_1)$, $g_2=gcd(i_5,i_6,i_2)$ and
$g_3=gcd(i_5,i_6,i_3)$. We know that $gcd(g_1,g_2,g_3) =1$ by our
choice of $d$. Theorem 4.1 in \cite{DZ} regarding intersection of Fermat hypersurfaces seems to suggest
that there can not be many solutions.
\medskip

In the particular case when $gcd(g_1,g_2)=1$  the analysis in
\cite{CL} tells us that $S_{(i_5,i_6,i_1)}$ and $S_{(i_5,i_6,i_2)}$
are coprime and each hypersurface defined by $S_{(i_5,i_6,i_1)}=0$
and $S_{(i_5,i_6,i_2)}=0$ in $\mathbb C^3$ has distinct reduced
irreducible components of dimension 2. Then their intersection $W$
has dimension 1. (Note that with respect to the setting of
\cite{CL}, we have fixed the first coordinate, hence we get one
dimension less). 
\medskip

In our case we need  only those solutions such that $|X_j| = 1, \,
\forall j $. In other words we need the set $W\cap \mathbb T^3$.
This condition in itself is very restrictive. In the previous
sections, where we used the two-circles argument along with mutual
orthogonality, we saw that this set can be finite. If an analysis as
in \cite{CL} can be carried through to get that $W\cap \mathbb T^3$
is indeed finite, we immediately get a bound on the period $d$. Then
along with orthogonality, one may be able to resolve the remaining
case of the $3$-intervals!

\end{document}